\documentclass[final]{siamltex}
\usepackage{amssymb}
\usepackage{amsmath, xcolor}
\usepackage{array,color,colortbl}
\newtheorem{example}[theorem]{Example}

\renewcommand{\geq}{\geqslant}
\renewcommand{\leq}{\leqslant}
\renewcommand{\ge}{\geqslant}
\renewcommand{\le}{\leqslant}

\newcommand{\per}{\mathop{\mathrm{per}}}
\def\half{\frac{1}{2}}
\def\RPA{\textrm{RPA}}
\def\LRPA{\textrm{LRPA}}

\def\eref#1{$(\ref{#1})$}
\def\sref#1{Section~$\ref{#1}$}
\def\lref#1{Lemma~$\ref{#1}$}
\def\tref#1{Theorem~$\ref{#1}$}

\providecommand{\cnit}{\cellcolor[gray]{.8}}

\providecommand{\LL}{\mathcal{D}}
\providecommand{\LR}{\mathcal{E}}

\title{On the Existence of Retransmission Permutation Arrays}

\author{Ian M. Wanless\footnotemark[2]
           \and Xiande Zhang\footnotemark[2]}

\begin{document}

\maketitle
\renewcommand{\thefootnote}{\fnsymbol{footnote}}
\footnotetext[2]{School of Mathematical Sciences,
Monash University,
VIC 3800, Australia
 ({\tt \{ian.wanless, xiande.zhang\}@monash.edu
}). This research was supported by ARC grants DP1093320
and DP120100197.}
\renewcommand{\thefootnote}{\arabic{footnote}}

\begin{abstract}
We investigate retransmission permutation arrays (RPAs) that are
motivated by applications in overlapping channel transmissions. An RPA
is an $n\times n$ array in which each row is a permutation of $\{1,
\ldots , n\}$, and for $1\leq i\leq n$, all $n$ symbols occur in each
$i\times\left\lceil\frac{n}{i}\right\rceil$ rectangle in specified
corners of the array. The array has types 1, 2, 3 and 4 if the stated
property holds in the top left, top right, bottom left and bottom
right corners, respectively. It is called latin if it is a latin
square. We show that for all positive integers $n$,
there exists a type-$1,2,3,4$ $\RPA(n)$ and a type-$1,2$ latin
$\RPA(n)$.

\end{abstract}

\begin{keywords}
retransmission permutation array, latin rectangle, latin square
 
\end{keywords}

\begin{AMS}
05B15
\end{AMS}

\pagestyle{myheadings} \thispagestyle{plain} \markboth{I. WANLESS AND X. ZHANG}{RETRANSMISSION PERMUTATION ARRAYS}

\section{Introduction}\label{s:intro}
Collisions in overlapping OFDM channels are becoming an increasingly
major problem in the deployment of high-speed wireless networks,
especially when the number of orthogonal channels is limited (e.g.,
\cite{ajss,bcmmw,cmmrb,gpbb,msba}). In an OFDM network, each channel
is allocated a set of consecutive subcarriers, and a message is
transmitted by assigning one bit to each subcarrier.  Two channels
overlap when they contain common subcarriers. When two messages are
transmitted in overlapping channels, only the bits assigned to the
subcarriers contained by both channels collide; the bits in other
subcarriers are clean and can be collected. However, the non-colliding
subcarriers do not contain complete message information. The
consequence is that messages will need to be retransmitted in order to
get all bits successfully.

Retransmission permutation arrays were introduced by
Li~et.~al~\cite{lltvy09} in an attempt to resolve such overlapping
channel retransmission problems. The main idea is to use a different
assignment of bits to subcarriers in each retransmission in order that
the number of retransmissions is optimized for overlapping channels
with any number of colliding subcarriers.  Since each assignment can
be viewed as a permutation of the set of bits, the schedule of
assignments can be arranged as a permutation array.  This motivated
Dinitz~et.~al~\cite{dpsw} to study the following new type of
combinatorial structures.

A {\em type $1$ retransmission permutation array of order $n$},
denoted {\em type-$1$ $\RPA(n)$}, is an $n\times n$ array $A$, in which
each cell contains a symbol from the set $[n]=\{1,\ldots,n\}$, such
that the following properties are satisfied:
\begin{itemize}
\item[(i)] every row of $A$ is a permutation of the $n$ symbols,
and
\item[(ii)] for $1\leq i \leq n$, the
  $i\times\left\lceil\frac{n}{i}\right\rceil$ rectangle in the upper
  left hand corner of $A$ contains all $n$ symbols.
\end{itemize}
There are variations of the above definition. If property (ii) is
modified so it instead holds for rectangles in the upper right,
lower left, or lower right hand corner of $A$, then we say that $A$
is, respectively, a {\em type $2$}, {\em type $3$} or {\em type $4$}
$\RPA(n)$.

An array that is simultaneously a type-$1$ $\RPA(n)$ and a type-$2$
$\RPA(n)$ is referred to as a type-$1,2$ $\RPA(n)$. This notation will
be generalized in the obvious way to arrays that satisfy different
combinations of the variations of property (ii).  The different types of
$\RPA$s are closely related, as shown in \cite{dpsw}. The same paper
showed that, for all positive integers $n$, there exists a type-$1,2$
$\RPA(n)$, a type-$1,3$ $\RPA(n)$ and a
type-$1,4$ $\RPA(n)$.  The following is an example of a type-$1,2,3,4$
$\RPA(11)$.

\begin{example}\label{lrpa11}
  A type-$1,2,3,4$ $\RPA(11)$, where the cells in the $i\times
  \left\lceil \frac{11}{i}\right\rceil$ rectangle, $1\leq i \leq 11$,
  in the four corners are shaded.
\begin{align*}
\centering \footnotesize
 \begin{array}{|c|c|c|c|c|c|c|c|c|c|c|}
\hline
\cnit6 &\cnit2& \cnit3&\cnit4&\cnit5&\cnit1&\cnit7&\cnit8&\cnit9&\cnit10&\cnit11\\
\hline
\cnit7&\cnit11&\cnit8&\cnit9&\cnit10&\cnit6&\cnit2&\cnit3&\cnit4&\cnit1&\cnit5\\
\hline
\cnit10&\cnit5&\cnit1&\cnit11&3&4&8&\cnit9&\cnit6&\cnit7&\cnit2\\
\hline
\cnit4&\cnit9&\cnit11&1&7&5&6&2&\cnit10&\cnit3&\cnit8\\
\hline
\cnit3&\cnit8&\cnit5&7&6&10&1&11&\cnit2&\cnit4&\cnit9\\
\hline
\cnit1&\cnit7&10&2&9&8&3&4&5&\cnit11&\cnit6\\
\hline
\cnit9&\cnit4&\cnit6&10&11&2&5&1&\cnit7&\cnit8&\cnit3\\
\hline
\cnit8&\cnit3&\cnit2&6&4&9&11&7&\cnit1&\cnit5&\cnit10\\
\hline
\cnit2&\cnit1&\cnit7&\cnit5&8&3&10&\cnit6&\cnit11&\cnit9&\cnit4\\
\hline
\cnit5&\cnit6&\cnit4&\cnit3&\cnit1&\cnit11&\cnit9&\cnit10&\cnit8&\cnit2&\cnit7\\
\hline
\cnit11&\cnit10&\cnit9&\cnit8&\cnit2&\cnit7&\cnit4&\cnit5&\cnit3&\cnit6&\cnit1\\
\hline
  \end{array}
\end{align*}
\end{example}

An infinite family of type-$1,2,3,4$ $\RPA$s exists by \cite{dpsw}:

\begin{theorem}\label{t:type1234rpa}
For all positive even integers $n$, there exists a type-$1,2,3,4$ $\RPA(n)$.
\end{theorem}

A retransmission permutation array $A$ is called {\em
  latin} if every column of $A$ is a permutation of the $n$ symbols. A
latin retransmission permutation array ($\LRPA$) of order $n$ is in fact a {\em
  latin square}, where each symbol occurs exactly once in each row and
column. Note that the type-$1,2,3,4$ $\RPA(11)$ in Example~\ref{lrpa11}
is actually latin. For small orders
$n\in\{1,2,\ldots,9,10,12,14,16,36\}$, a type-$1,2,3,4$ $\LRPA(n)$ was
constructed in \cite{dpsw}. Finding general constructions for $\LRPA$s
seems to be quite difficult.  However, it was conjectured in \cite{dpsw}
that a type-$1,2,3,4$ $\LRPA(n)$ exists for all positive integers $n$.

The main contribution of this paper is new constructions for $\RPA$s
and $\LRPA$s. In particular, in \sref{s:type1234} we strengthen
\tref{t:type1234rpa} by showing that a type-$1,2,3,4$ $\RPA(n)$ exists
for all positive integers $n$. Further, in \sref{s:ltype12} we prove
that a type-$1,2$ $\LRPA(n)$ exists for all positive integers $n$,
which is the first known infinite family of latin $\RPA$s. 

\section{Type-$1,2,3,4$ $\RPA$s of odd orders}\label{s:type1234}

Theorem~2.7 of \cite{dpsw} gave an effective way to construct a
type-$1,2,3,4$ $\RPA$ of even order from a type-$1,2$ $\RPA$. In this
section, we will use a similar idea to construct type-$1,2,3,4$
$\RPA$s of odd order by imposing some additional structure on the
type-$1,2$ $\RPA$s.
Note that Dinitz et al.~\cite[Figure 3]{dpsw} gave an algorithm to
construct type-$1,2$ $\RPA$s of odd orders.  We adapt their algorithm
to construct the desired type-$1,2$ $\RPA$s with extra conditions. To
make the paper more self-contained, we now define some terminology and
notation used in the algorithm.

For the time being, we assume that $n$ is odd.
An $r\times \left\lceil \frac{n}{r}\right\rceil$ rectangle is called
{\em basic} if it does not contain an $r'\times \left\lceil
  \frac{n}{r'}\right\rceil$ rectangle where $r'<r$ and $\left\lceil
  \frac{n}{r}\right\rceil=\left\lceil \frac{n}{r'}\right\rceil$. When
verifying property (ii) in the definition of $\RPA$s, it suffices to
consider only basic rectangles. A basic rectangle $R$ of size $r\times
\left\lceil \frac{n}{r}\right\rceil$ is called {\em canonical} if
there is a partition $n=a_1+\cdots+a_r$ where $\left\lceil
  \frac{n}{r}\right\rceil= a_1\geq \cdots\geq a_r>0$, such that
every symbol in $[n]$ occurs exactly once in the
union, over $1\leq i\leq r$, of the first $a_i$ cells of row $i$.

\begin{example}\label{lrpa13}
  A type-$1,2,3,4$ $\LRPA(13)$. The
  basic rectangles in the upper left hand corner have dimensions
  $1\times13$, $2\times7$, $3\times 5$, $4\times4$, $5\times3$,
  $7\times2$ and $13\times 1$. The shaded symbols show that
  these basic rectangles are canonical.
\begin{align*}  \centering \footnotesize
\begin{array}{|c|c|c|c|c|c|c|c|c|c|c|c|c|}
\hline
\cnit7& \cnit2& \cnit3& \cnit4& \cnit5& \cnit6& \cnit1& \cnit8& \cnit9& \cnit10&\cnit11&\cnit12&\cnit13\\ \hline
\cnit8& \cnit13&\cnit12&\cnit11&\cnit10&\cnit9& 7& 5& 4& 3& 2& 1& 6\\ \hline
\cnit9& \cnit6& \cnit1& 2& 3& 4& 10&11&12&13&7& 8& 5\\ \hline
\cnit10&\cnit5& 6& 1& 2& 3& 11&7& 13&12&8& 9& 4\\ \hline
\cnit11&\cnit4& 5& 6& 1& 2& 12&13&8& 7& 9& 10&3\\ \hline
\cnit12&\cnit3& 4& 5& 7& 13&9& 1& 6& 8& 10&11&2\\ \hline
\cnit1& 8& 9& 12&6& 10&2& 3& 11&5& 4& 13&7\\ \hline
\cnit2& 11&10&13&8& 7& 5& 4& 1& 9& 6& 3& 12\\ \hline
\cnit3& 10&7& 9& 13&8& 6& 12&2& 1& 5& 4& 11\\ \hline
\cnit4& 9& 13&8& 11&12&3& 2& 5& 6& 1& 7& 10\\ \hline
\cnit5& 1& 8& 7& 12&11&4& 10&3& 2& 13&6& 9\\ \hline
\cnit6& 7& 2& 3& 4& 1& 13&9& 10&11&12&5& 8\\ \hline
\cnit13&12&11&10&9& 5& 8& 6& 7& 4& 3& 2& 1\\ \hline
  \end{array}
\end{align*}
\end{example}

Denote the number of basic rectangles by $b$. For $1\leq k \leq b$,
let the $k$-th basic rectangle be $R_k$ and suppose it has dimensions
$r_k\times c_k$.
For $2\leq k\leq b-2$, a canonical basic rectangle
$R_k$ is called {\em sum-free} if the following two conditions are
satisfied.
\begin{itemize}
 \item[(a)] For any two symbols $x,y$ in the same row of $R_k$, it
   holds that $x+y\neq n+1$.
 \item[(b)] For any two symbols $x,y$ in the last $c_k-c_{k+1}$
   columns of $R_k$, it holds that $x+y\neq n+1$.
\end{itemize}

Following a similar strategy to \cite{dpsw}, we work on
$R_1,R_2,\dots,R_b$ in turn. At the point at which we start working on
$R_k$ for $3\le k\le b-1$, it should be true that $R_{k-1}$ is a
sum-free canonical basic rectangle (although $R_{k-1}$ may
subsequently be changed).  We then copy the entries in $R_{k-1}\setminus
R_k$ into empty cells in $R_k$, in a way that ensures $R_k$ is
canonical and satisfies Condition (a).  Our ability to do this is
guaranteed by \cite[Lemma~4.6]{dpsw}.  Next, we perform {\em simple
  exchange} (SE) {\em operations} to ensure that $R_k$ is sum-free. An
SE operation swaps the contents of two non-empty cells within the same
row. It is clear that SE operations within $R_k$ cannot affect whether
$R_1,\dots,R_k$ satisfy Condition (a).  Details of how we achieve Condition (b)
will be given below. Once the top left hand corner has been settled,
in each row except the first we will ``reflect'' the filled entries to
produce basic rectangles in the top right corner.  By construction,
$R_2,\dots,R_b$ will all satisfy Condition (a), which guarantees that
the reflection does not repeat any symbol within a row, so it is then
trivial to complete the array to a type-1,2 $\RPA$.

As just mentioned, the crucial step is to show that SE operations can
modify a given canonical basic rectangle $R_k$ to make it sum-free.
Partition $R_k$ into two parts $\LL_k$ and $\LR_k$, where $\LL_k$
consists of the first $c_{k+1}$ columns of $R_k$ and $\LR_k$ consists
of the remaining columns of $R_k$. Define a graph $G_k$ over the $n$
non-empty cells of $R_k$, where two cells are adjacent if the symbols
in the two cells sum to $n+1$. Note that $G_k$ consists of $(n-1)/2$
disjoint edges, and no two cells in a row form an edge if $R_k$
satisfies Condition (a). We say an edge of $G_k$ is of {\it type A} if
both endpoints are in $\LR_k$, of {\it type B} if exactly one endpoint
is in $\LR_k$ and of {\it type C} if no endpoint is in $\LR_k$. In
order to satisfy Condition (b), we will remove type A edges from $G_k$
by SE operations.  The description thus far mirrors the algorithm used
in \cite{dpsw}.  The main difference in our approach is that some
additional conditions will be imposed on the $\RPA$, which we now
construct.

\begin{lemma}\label{type1,2}
  Let $n\geq 15$ be an odd integer and $h=(n+1)/2$. There exists a
  type-$1,2$ $\RPA(n)$, $A=(a_{i,j})$, with the following properties:
\begin{itemize}
 \item $a_{1,1}=h$, $a_{h,1}=1$ and $a_{h,n}=h$;
 \item $a_{i,n}\neq 1$ for $1\leq i \leq h$;
 \item $A$ contains a canonical basic rectangle of size
   $h\times 2$ in the upper left and right corners.
\end{itemize}
\end{lemma}
\begin{proof}Since $n\geq 15$, we have $c_2-c_3\geq 2$ and
$r_{b-1}-r_{b-2}\geq 2$. For $1\leq k\leq b$, we will write the symbols
from $[n]$ into $R_k$ as described before the lemma.

When $k=1$, fill in $R_1$ of size $1\times n$, from left to right with
symbols $1,2,\ldots,n$. Then interchange the symbols $1$ and $h$.

When $k=2$, fill in the first $h-1$ cells of row $2$ from left to
right with the symbols $h+1,n,h+2,h+3,\ldots,n-1$. It is easy to check that
$R_2$ is a sum-free canonical basic rectangle.

When $k=3$, copy the symbols in $ \LR_2$ into the left most cells of
row $3$. Since $1\in \LR_2$ and $|\LR_2|\geq 3$, we can set
$a_{3,3}=1$. It is obvious that $R_3$ is canonical and satisfies Condition (a).

For $3\leq k \leq b-2$, suppose we have a canonical basic rectangle
$R_k$ that satisfies Condition (a).  Let $a_k$ denote the number of
type A edges in $G_k$. If $a_k=0$ then $R_k$ is sum-free, so copy the
symbols in $\LR_{k}$ into the empty cells of $R_{k+1}$ such that
$R_{k+1}$ is canonical and satisfies Condition (a). This is guaranteed
to be possible by \cite[Lemma~4.6]{dpsw} and the remark following it.
If $a_k>0$, we will show that $a_k$ can be reduced by at least one by
applying SE operations to cells other than $(1,1)$, $(2,2)$ and
$(3,3)$. Therefore a sequence of SE operations will reduce $a_k$ to
$0$ with $a_{1,1}=h$, $a_{2,2}=n$ and $a_{3,3}=1$ fixed.

Now we deal with the case when $a_k>0$. For convenience, let $F$ be
the set of cells $\{(2,2),(3,3)\}$. Let $x_1y_1$ be a type A edge with
$x_1$, $y_1$ in rows $i_1$, $i_2$ respectively. Let $I$ denote the set
of cells in rows $i_1$ and $i_2$. If there is a type C edge $x_2y_2$
where $x_2\in I$ and $\{x_2,y_2\}\cap F=\emptyset$, then
stop. Otherwise, all edges with one endpoint in $(I\cap
\LL_k)\setminus F$ are of type B. However,
\[
|I\cap \LL_k\setminus F|-|I\cap \LR_k\setminus \{x_1,y_1\}|
\geq|I\cap \LL_k|-|I\cap\LR_k|>0.
\]
The last inequality holds by \cite[Lemma~4.2]{dpsw}. Therefore, one of
these type B edges has an endpoint in $\LR_k\setminus I$. Denote the
edge by $x_2y_2$, where $x_2\in (I\cap \LL_k)\setminus F$ and $y_2$ is
in row $i_3$ of $\LR_k\setminus I$. Here $y_2\not\in F$ since
$(2,2)(3,3)$ is an edge of $G_k$. Add row $i_3$ to $I$.  If there is a
type C edge $x_3y_3$ where $x_3\in I$ and $\{x_3,y_3\}\cap
F=\emptyset$, then stop. If not, since
\[
|I\cap \LL_k\setminus(F\cup\{x_2\})|-|I\cap \LR_k\setminus \{x_1,y_1,y_2\}|
\geq|I\cap\LL_k|-|I\cap \LR_k|>0,
\]
there is a type B edge $x_3y_3$, where $\{x_3,y_3\}\cap F=\emptyset$,
$x_3\in I\cap \LL_k$ and $y_3$ is in row $i_4$ of $\LR_k\setminus
I$. Then add row $i_4$ to $I$ and continue this process until we find
a type C edge $x_j y_j$ where $x_j\in I\cap \LL_k$ and
$\{x_j,y_j\}\cap F=\emptyset$, which must happen since $|I|$ is
bounded by $r_k$ and thus cannot grow indefinitely.  Now, imitating
\cite[Lemma~4.5]{dpsw}, we find a sequence $e_1,\dots,e_\ell$ of edges
from among $\{x_1y_1,\ldots,x_jy_j\}$. We start with $e_1=x_jy_j$,
which is of type C.  Then for $1<i<\ell$ the type B edge $e_i$ is
chosen so that its vertex in $\LR_k$ is in the same row as the vertex
of $e_{i-1}$ in $\LL_k\cap I$. Eventually we reach a type B edge
$e_{\ell-1}$ with a vertex in $i_1$ or $i_2$. This determines $\ell$,
and we finish with $e_\ell=x_1y_1$, which is of type A.  Crucially,
this construction allows us to perform SE operations to swap one end
of $e_i$ with an end of $e_{i+1}$ for $1\le i<\ell$. These operations
remove the type A edge $x_1y_1$ (all the affected edges are now of
type B), thereby reducing $a_k$ by one.  Whenever an SE operation is
performed on two cells in row $1$, say $(1,i)$ and $(1,j)$ with
$1<i<j<h$, we also perform an SE operation on cells $(1,n+1-i)$ and
$(1,n+1-j)$.  Thus after a sequence of SE operations, we can change a
canonical basic rectangle $R_k$ to a sum-free rectangle with
$a_{2,2}=n$ and $a_{3,3}=1$ fixed. Note that $a_{1,1}=h$ is fixed
since $(1,1)$ is an isolated vertex in $G_k$, for $3\leq k \leq b-2$.

When $k=b-1$, we can let $a_{h,1}=1$ since $1\in \LR_{b-2}$.  When
$k=b$, copy the symbols in $ \LR_{b-1}$ into the empty cells in $R_b$.

At last, define the reflection map $\pi:[n]\rightarrow[n]$ as follows:
\begin{align}\label{ppi}
\pi(i)=
\begin{cases}
n+1-i&\text{if $i\not \in\{1,h\}$},\\
h&\text{if }i=1,\\
n&\text{if }i=h.\\
\end{cases}
\end{align}
For each non-empty cell $(i,j)$ in rows $2,\ldots,n$ of $A$ define
$a_{i,n+1-j}=\pi(a_{i,j})$. Complete each row of $A$ to a permutation
of the $n$ symbols to get the desired type-$1,2$ $\RPA(n)$. Note that
$a_{2,2}=n$, so that $a_{i,1}\neq n$ and hence
$a_{i,n}\neq 1$ for $1\leq i \leq h$.
\end{proof}

\begin{lemma}\label{type1,234}
For all odd integers $n$, there is a type-$1,2,3,4$ $\RPA(n)$.
\end{lemma}

\begin{proof}
For $n\leq 13$, a type-$1,2,3,4$ $\RPA(n)$ exists by \cite{dpsw} and
Examples \ref{lrpa11} and \ref{lrpa13}. For $n\geq 15$, let
$A=(a_{i,j})$ be the type-1,2 $\RPA(n)$ on $[n]$ obtained in
\lref{type1,2}.
Let $X=\{a_{2,1},\ldots, a_{h-1,1}\}$,
$Y=[n]\setminus (X\cup\{1,h\})$,
$Z=\{a_{1,n},\ldots, a_{h-1,n}\}$ and
$W=[n]\setminus (Z\cup\{1,h\})$. Let $\sigma:[n]\rightarrow[n]$
be any bijection
satisfying $\sigma(X)=W$, $\sigma(Z)=Y$, $\sigma(1)=h$ and
$\sigma(h)=1$ (such a bijection must exist because $|X|=h-2=|W|$,
$|Z|=h-1=|Y|$ and $|X\cap Z|=|X|-|X\cap W|=|W|-|X\cap W|=|W\cap Y|$).
Define an $n\times n$ array $B=(b_{i,j})$ as follows:
\begin{equation*}
b_{i,j}=\begin{cases}
a_{i,j} &\text{if $i\leq h$,}\\
\sigma(a_{n+1-i,n+1-j}) &\text{if $i>h$.}
\end{cases}
\end{equation*}

It can be checked that $B$ is a type-$1,2,3,4$ $\RPA(n)$.  Most of the
verifications are straightforward. The only tricky points are to
ensure that the $n\times 1$ and $h\times 2$ basic rectangles in the
lower left and lower right corners of $B$ contain all $n$ symbols. The
$n\times 1$ basic rectangles in the lower left and lower right corners
of $B$ contain $X\cup \{h,1\}\cup Y$, $Z\cup W \cup \{1,h\}$
respectively, which are the whole set $[n]$. Let $U$ be the $h\times 2$ basic rectangle in the upper
left corner of $A$. The $h\times 2$ basic rectangle in the lower
right corner of $B$, say $V$, contains the symbol $h$ together with
the image of the first $h-1$ rows of $U$ under the bijection
$\sigma$. Because $U$ is canonical, it follows that the first $h-1$
rows of $U$ contain all the symbols in $[n]\setminus
\{1\}$, and hence the last $h-1$ rows of $V$ contain all the
symbols in $[n]\setminus \{h\}$. A similar argument
applies to the $h\times 2$ basic rectangle in the lower left corner of
$B$.
\end{proof}

Combining \lref{type1,234} and \tref{t:type1234rpa}, we get the
following result.

\begin{theorem}\label{type1,234rpa}
For all positive integers $n$, there exists a type-$1,2,3,4$ $\RPA(n)$.
\end{theorem}

\section{Type-$1,2$ $\LRPA$s}\label{s:ltype12}

In this section, we prove the existence of a type-$1,2$ $\LRPA(n)$ for
all positive integers $n$.  First we need some terminology for dealing
with latin squares. A {\em partial latin square} of order $n$ is an
$n\times n$ array with each cell filled or empty, such that each
symbol occurs at most once in each row and column.  A partial latin
square is {\em completable} if it can be completed to a latin
square. A row (column) in a partial latin square is {\em completable}
if it can be extended to another partial latin square by filling the
row (column).  A {\em latin rectangle} is a rectangular array in which
each symbol in a given set of symbols occurs exactly once in each row
and at most once in each column. It is convenient to also use the
phrase ``latin rectangle'' to describe a partial latin square in which
some rows are completely filled and the other rows are completely empty.

The following famous theorem is due to Hall~\cite{h45}.
\begin{theorem}\label{hall}
Any $r\times n$ latin rectangle is completable to an $n\times n$ latin square.
\end{theorem}

Define $\Xi_{k,n,\ell}$ to be the set of $n\times n$ partial latin
squares with the first $k$ rows filled, exactly $\ell$ filled cells in
row $k+1$, and no other filled cells. Then define $f_{k,n}$ to be the
maximum integer $\ell<n$ such that all partial latin squares in
$\Xi_{k,n,\ell}$ are completable.  Note that, by \tref{hall}, partial
latin squares in $\Xi_{k,n,\ell}$ are completable if and only if row
$k+1$ is completable. Informally, $f_{k,n}$ is the number of
``free-choices" that you have when extending a $k\times n$ latin
rectangle by one row. Due to \tref{hall} the value of $f_{k,n}$ is
well-defined, and we have $f_{k,n}\geq 1$ for $0 <k < n-1$ and
$f_{0,n}=f_{n-1,n}=n-1$. It has been shown in \cite{brualdicsima} that 
we often have
many free choices, particularly when extending ``thin'' latin
rectangles:

\begin{theorem}\label{l:fkn}
\[
f_{k,n}=
\begin{cases}
n-2k&1\le k<\half n,\\
1&\half n\le k\le n-2.
\end{cases}
\]
\end{theorem}

The values of $f_{k,n}$ were determined in \cite{brualdicsima} by applying the classical Frobenius-K\"onig Theorem. We restate the main idea of the proof below in Lemma~\ref{l:FrobKon} since it is useful to construct type-$1,2$ $\LRPA$s.
For an $n\times n$ matrix $M = (m_{ij})$, the {\em permanent} of $M$,
$\per M$, is defined by
\begin{align*}
\per M =\sum_{\tau} m_{1\tau(1)}\cdots m_{n\tau(n)},
\end{align*}
where the sum is over all permutations $\tau$ of $[n]$.

Suppose $S\in\Xi_{k,n,\ell}$ and let $R$ be the $k\times n$ latin
rectangle formed by the first $k$ rows of $S$.  Define a
$(0,1)$-matrix $M(R)$ by assigning the $(i,j)$-entry to be $1$ if and only
if symbol $i$ does not appear in column $j$ of $R$. Let $M'(S)$ be the
submatrix of $M(R)$ obtained by deleting the $\ell$ rows
and $\ell$ columns that correspond to, respectively, the symbols and columns
of the filled entries in row $k+1$ of $S$. We have:

\begin{lemma}\label{l:FrobKon}
$S$ is not completable if and only if $M'(S)$ contains an $r\times s$
matrix of zeros where $r+s=n-l+1$.
\end{lemma}

\begin{proof}
It is well known that the number of extensions of $R$ to a
$(k+1)\times n$ latin rectangle is $\per M(R)$.  By the same logic,
the number of extensions of $S$ to a $(k+1)\times n$ latin rectangle
is the permanent of the $(n-l)\times(n-l)$ submatrix $M'(S)$.
In particular, $S$ is not completable if and only if $\per M'(S)=0$.
By the Frobenius-K\"onig Theorem, $\per M'(S)=0$
if and only if $M'(S)$ contains an $r\times s$ zero submatrix
such that $r+s=n-l+1$.
\end{proof}

We will apply Theorem~\ref{l:fkn} and Lemma~\ref{l:FrobKon}
repeatedly in our construction of $\LRPA$s. For each integer $n$, let
$A$ be a type-$1,2$ $\RPA(n)$ constructed in \cite{dpsw}. Denote by $P$
the partial matrix obtained from $A$ by selecting the cells in the
first $b-1$ canonical basic rectangles in the upper left corner, and
the corresponding cells in the upper right corner of $A$.  Note that
no symbol repeats in a column of $P$ because the basic rectangles are
canonical. Since the rows of $A$ are permutations, $P$ is a partial
latin square. Now we will complete $P$ to a type-$1,2$ $\LRPA(n)$.

\begin{theorem}\label{type1,2lpra}
For all positive integers $n$, there exists a type-$1,2$ $\LRPA(n)$.
\end{theorem}

\begin{proof}
For $n<14$, a type-$1,2$ $\LRPA(n)$ exists by \cite{dpsw} and
Examples \ref{lrpa11} and \ref{lrpa13}, so we may assume $n\ge14$.

The first two rows of $P$ are complete, except for the empty cell
$(2,h)$ when $n$ is odd. If $n$ is odd, suppose the cell $(1,h)$
is filled with symbol $z$.  Then fill $(2,h)$ with $\pi(z)$, where
$\pi$ is defined in \eref{ppi}. It is easy to check that row $2$ is
also a permutation.

Now we fill in $P$ one row at a time.
Suppose that we have completed
$k$ rows, where $2 \leq k < n/2$. Then $P$ contains at most
$2\left\lceil n/(k+1)\right\rceil$ filled cells in row $k+1$.
So, using Theorem~\ref{l:fkn}, we can fill row $k+1$ provided
\begin{equation}\label{e:keycond}
n-2k\ge2\left\lceil\frac{n}{k+1}\right\rceil.
\end{equation}
Checking directly, this condition holds when $k\in\{2,3\}$,
given that $n\ge14$. So we may assume $k\ge4$.

Suppose that $k\le(n-4)/3$. If \eref{e:keycond} fails,
then
\[
n-2k<2\left\lceil\frac{n}{k+1}\right\rceil
\le\frac{2(n+k)}{k+1}.
\]
Rearranging, we find that
$2k(k+2)>n(k-1)\ge(3k+4)(k-1)$, which contradicts $k\ge4$.

So we may assume that $k\ge n/3-1$. In this case,
$2\left\lceil n/(k+1)\right\rceil\le6$, so \eref{e:keycond} holds
for $k\le\half(n-6)$.

Since $n\ge14$, there are exactly four filled cells
in row $k+1$ of $P$ if $\half(n-6)<k<\half(n-1)$, and exactly
two filled cells if $k=\half(n-1)$.
For $k\in\big\{\half(n-5),\half(n-4)\big\}$ this is sufficient, since
$n-2k\ge4$.

If $k=\half(n-2)$, then let $R$ be the $k\times n$ latin rectangle
contained in $P$.  In this case, $M(R)$ has exactly $k$ zeros in each row and
column. Denote the four distinct symbols in row $k+1$ of $P$ by
$a,b,c,d$. Let $M'$ be the submatrix obtained from $M(R)$ by deleting
rows $\{a,b,c,d\}$ and columns $\{1,2,n-1,n\}$. Since each basic
rectangle of size $\frac{n}{2}\times 2$ in the upper left and upper
right corners covers all symbols exactly once, each row of $M'$ has
exactly $k-2$ zeros. Thus any $r\times s$ zero submatrix of $M'$ has
$r\leq k$ and $s\leq k-2$, meaning that $r+s\leq k+k-2= n-4<n-3$. By
\lref{l:FrobKon}, row $n/2$ is completable.

It remains to consider the possibility that
$k\in\big\{\half(n-3),\half(n-1)\big\}$ and we cannot complete
row $k+1$. By \lref{l:FrobKon}, $M(R)$ contains a $r\times s$ zero
submatrix with $r+s=n-l+1=2k$. As $M(R)$ has only $k$ zeros in each
row and column, we must have $r=s=k$. This means that $R$ contains a
$k\times k$ latin subsquare, say $C$.  Consider the first three
columns of $R$. Either two of them are in $C$, or else two of them are
outside of $C$. In either case there are two columns that between them
contain at most $n-k$ distinct symbols. Hence the $\left\lceil
  n/3\right\rceil\times3$ rectangle in the top left corner of $R$
contains at most $n-k+\left\lceil n/3\right\rceil\le
n-\half(n-3)+\left\lceil n/3\right\rceil<n$ distinct symbols, using
$n\ge14$. This contradicts the choice of $P$.

We conclude that we can complete the first
$\left\lceil n/2\right\rceil$ rows of $P$.
Finally, we complete $P$ to a latin square
by \tref{hall} to get a type-$1,2$ $\LRPA$.
\end{proof}

\section{Concluding remarks}\label{S:concl}

The study of retransmission permutation arrays was motivated by a
technique used to resolve overlapping channel transmissions.  We
completed the existence spectrum for type-$1,2,3,4$ $RPA$s and
constructed a type-$1,2$ $\LRPA(n)$ for all positive integers $n$,
which is the first known infinite family of latin $RPA$s.  By rotating
these examples, it is obvious that we obtain type-$2,4$ $\LRPA(n)$,
type-$3,4$ $\LRPA(n)$ and type-$1,3$ $\LRPA(n)$ for all positive
integers $n$. However, the existence of type-$1,2,3,4$ $\LRPA(n)$ for
all $n$, conjectured in \cite{dpsw}, remains open.

\end{document}